\documentclass[english,reqno]{amsart}
\usepackage{hyperref}
\usepackage[margin=1in]{geometry}
\usepackage{graphicx}
\usepackage{subfig}
\usepackage{amssymb,amsbsy,amsmath,amsfonts,amssymb,amscd}
\usepackage{comment}
\usepackage{verbatim}
\usepackage[dvipsnames]{xcolor}
\usepackage{enumitem}
\usepackage[numbers]{natbib} 
\usepackage{algorithm,algpseudocode}
\usepackage{setspace}
\newtheorem{theorem}{Theorem}[section]
\newtheorem{proposition}{Proposition}[section]
\newtheorem{remark}{Remark}[section]
\newtheorem{lemma}{Lemma}[section]

\newtheorem{definition}{Definition}[section]

\newcommand{\RN}[1]{%
  \textup{\uppercase\expandafter{\romannumeral#1}}%
}

\allowdisplaybreaks

\title{Rate of Convergence in Periodic Homogenization for Convex Hamilton--Jacobi Equations with multiscales}
\author{Yuxi Han and Jiwoong Jang}


\begin{document}
\maketitle

\begin{abstract}
We study the rate of convergence in periodic homogenization for convex Hamilton--Jacobi equations with multiscales, where the Hamiltonian $H=H(x, y, p): \mathbb{R}^n \times \mathbb{T}^n \times \mathbb{R}^n \to \mathbb{R
}$ depends on both of the spatial variable and the oscillatory variable. In particular, we show that for the Cauchy problem, the rate of convergence is $O(\sqrt{\epsilon})$ by optimal control formulas, scale separations and curve cutting techniques. We also show the rate $O(\sqrt{\epsilon})$ of homogenization for the static problem based on the same idea. Additionally, we provide examples that illustrate the rate of convergence for the Cauchy problem is optimal.
\end{abstract}

\section{Introduction}
We consider the periodic homogenization problem for convex Hamilton--Jacobi equations in the multiscale setting. For $\epsilon > 0$, let $u^\epsilon$ be the unique viscosity solution to 

\begin{equation}\label{eqn:ms}
    \left\{\begin{aligned}
    u^\epsilon_t+H \left(x, \frac{x}{\epsilon}, Du^\epsilon \right) & = 0 \quad \qquad \, \text{in } \mathbb{R}^n \times (0, \infty), \\
    u^\epsilon (x, 0) & = g(x)  \qquad \text{on } \mathbb{R}^n,
    \end{aligned}
    \right.
\end{equation}
where $g$ is a given function as the initial data and the Hamiltonian $H=H(x,y,p):\mathbb{R}^n\times\mathbb{T}^n\times\mathbb{R}^n \to \mathbb{R}$ is continuous and convex in $p$. Here, $\mathbb{T}^n=\mathbb{R}^n/\mathbb{Z}^n$ is the $n$-dimensional flat torus. It is well known that under appropriate assumptions, $u^\epsilon$ converges uniformly to the unique viscosity solution $u$ to 
\begin{equation}\label{eqn:eh}
    \left\{\begin{aligned}
    u_t+\overline{H} \left(x, Du\right) & = 0 \quad \qquad \, \text{in } \mathbb{R}^n \times (0, \infty), \\
    u (x, 0) & = g(x)  \qquad \text{on } \mathbb{R}^n,
    \end{aligned}
    \right.
\end{equation}
on $\mathbb{R}^n \times [0, T]$ for any $T>0$ as $\epsilon\to 0$, where $\overline{H}$ is the effective Hamiltonian of $H$ (see \cite{LionsPapaVara1987}). However, the optimal rate of convergence of $u^\epsilon$ to $u$ in this multiscale setting has not been studied thoroughly in the literature. In this paper, we prove that the rate of convergence is $O (\sqrt{\epsilon})$ for $t \geq \sqrt{ \epsilon}$ and $O \left(\min \left\{t, \epsilon\right\}\right)$ for $t \in \left(0, \sqrt{\epsilon}\right)$. Furthermore, examples are provided to demonstrate the optimality of this convergence rate.

\subsection{Relevant Literature}

Periodic homogenization for coercive Hamilton--Jacobi equations was first proved in \cite{LionsPapaVara1987}. Subsequently, numerous works in the literature have focused on determining the rate of convergence of the homogenization problem for Hamilton--Jacobi equations. For general nonconvex Hamiltonians with multiscales, the best known rate of convergence is $O(\epsilon^{1/3})$, which was obtained in \cite{ItaloIshii2001} by the doubling variable method and the perturbed test function method (see \cite{Evans1989,Evans1992}). For convex Hamiltonians $H=H(y, p)$ that depend only on the oscillatory variable and the momentum, the optimal rate of convergence was first studied in \cite{MitakeTranYu2019} using weak KAM theory and Aubry-Mather theory. In particular, it was proved that the lower bound of $u^\epsilon - u \geq -C \epsilon$ is optimal and the upper bound holds with additional assumptions on $H, u, g$. Recently, the optimal rate of $O(\epsilon)$ was proved in \cite{TranYu2022} using a curve cutting lemma from
metric geometry (see \cite{Burago1992}), which concludes the study in the setting of convex Hamiltonians $H=H(y, p)$ that depend only on the oscillatory variable and the momentum. Additionally, the optimal rate of $O(\epsilon)$ was obtained in \cite{Nguyen2022} for convex Hamiltonians $H=H(y,s,p)$ that also depend periodically on the time variable. For a recent study on the rate of convergence for time-fractional Hamilton--Jacobi equations with Caputo fractional derivatives, see \cite{MitakeSato2023}. We refer the reader to \cite{ItaloIshii2001,MitakeTranYu2019,TranYu2022} for further references therein.

To our best knowledge, the most closely related previous research in this area is \cite{Son2020}, where the approach in \cite{MitakeTranYu2019} was extended to attain the optimal rate of $O(\epsilon)$ in one dimension with further assumptions on $H$. In this study, we investigate this problem for dimensions $n \geq 1$ and prove that the convergence rate, in general, is $O (\sqrt{\epsilon})$ for $t \geq \sqrt{ \epsilon}$.

\subsection{Settings}
Throughout this paper, we will assume that the following conditions hold for the Hamiltonian $H : \mathbb{R}^n \times \mathbb{T}^n \times \mathbb{R}^n \to \mathbb{R}$ :

(H1) For each $R>0$, $H \in \mathrm{BUC} \left(\mathbb{R}^n \times \mathbb{T}^n \times \mathrm{B}\left(0, R\right)\right)$, where $\mathrm{BUC}\left(\mathbb{R}^n \times \mathbb{T}^n \times \mathrm{B}\left(0, R\right)\right)$ stands for the set of bounded and uniformly continuous functions on $\mathbb{R}^n \times \mathbb{T}^n \times \mathrm{B}\left(0, R\right)$.

(H2) $\lim_{\left|p\right| \to \infty}\left(\inf_{x\in \mathbb{R}^n,y \in \mathbb{T}^n}H\left(x,y,p\right)\right) = +\infty$.

(H3) For each $x \in \mathbb{R}^n$ and $y\in \mathbb{T}^n$, the map $p \mapsto H(x, y, p)$ is convex.

(H4) There exists a constant $\mathrm{Lip}(H)>0$ such that $\left|H(x_1, y, p)-H(x_2, y, p)\right| \leq \mathrm{Lip}(H) |x_1-x_2|$, for any $x_1, x_2 \in \mathbb{R}^n$, $y \in \mathbb{T}^n$, and $p \in \mathbb{R}^n$.

We also assume $g \in \mathrm{BUC}(\mathbb{R}^n) \cap \mathrm{Lip} (\mathbb{R}^n)$.

We emphasize that condition (H4) is essential for the validity of our main result (as discussed in Remark \ref{rem:cond4}). In Section \ref{sec:example}, we present an example (refer to Proposition \ref{exp:nolip}) to demonstrate that in the absence of this condition, the rate of convergence of $u^\epsilon$ to $u$ as $\epsilon$ tends to zero cannot be bounded by $O(\sqrt{\epsilon})$.

The well-posedness of the equation \eqref{eqn:ms} has already been extensively studied. The classical theory of viscosity solutions can be used to demonstrate the existence and uniqueness of solutions to \eqref{eqn:ms} (see \cite{tran_hamilton-jacobi_2021}). Moreover, the solution $u^\epsilon$ is uniformly bounded and Lipschitz, which can be expressed as follows:
\begin{equation}\label{eqn:u_prior}
\left\|u^\epsilon_t\right\|_{L^\infty\left(\mathbb{R}^n \times [0, \infty)\right)}+\left\|Du^\epsilon\right\|_{L^\infty(\mathbb{R} \times [0, \infty))} \leq C_0,\quad \forall \epsilon>0\,,
\end{equation}
where $C_0>0$ is a constant that depends only on $H$ and $\left\|Dg\right\|_{L^\infty(\mathbb{R}^n)}$.  Based on \eqref{eqn:u_prior}, we can modify $H(x, y, p)$ for $|p| > 2C_0+1$ without changing the solutions to \eqref{eqn:ms}. This modification ensures that for all $x, p \in \mathbb{R}^n$ and $y \in \mathbb{T}^n$,
\begin{equation}\label{K_0H}
    \frac{|p|^2}{2}-K_0 \leq H(x, y, p) \leq \frac{|p|^2}{2}+K_0
\end{equation}
for some constant $K_0 >0$ that depends only on $H$ and $\left\|Dg\right\|_{L^\infty(\mathbb{R}^n)}$. Consequently, for all $x, v \in \mathbb{R}^n$ and $y \in \mathbb{T}^n$,
\begin{equation}\label{K_0L}
\frac{|v|^2}{2}-K_0 \leq L(x, y, v) \leq \frac{|v|^2}{2}+K_0
\end{equation}
where $L:\mathbb{R}^n \times \mathbb{T}^n \times \mathbb{R}^n \to \mathbb{R
}$ is the Legendre transform of $H$.

Moreover, we have optimal control formulas for $u^\epsilon$ and $u$, that is,
\begin{equation}\label{eqn:ocfue}
    \begin{aligned}
        u^\epsilon (x, t)&= \inf \left\lbrace \int_0^t L\left(\gamma(s), \frac{\gamma(s)}{\epsilon} ,-\dot{\gamma}(s)\right) ds+g\left(\gamma(t)\right): \gamma\in \mathrm{AC}([0,t];\mathbb{R}^n),\ \gamma(0) = x\right\rbrace
    \end{aligned}
\end{equation}
and
\begin{equation}\label{eqn:ocfuebar}
    \begin{aligned}
        u(x, t) = \inf \left\lbrace \int_0^t \overline{L} \left( \overline{\gamma}(s), -\dot{\overline{\gamma}}(s)\right) ds+g\left(\overline{\gamma}(t)\right): \overline{\gamma}\in \mathrm{AC}([0,t];\mathbb{R}^n),\ \overline{\gamma}(0) = x\right\rbrace,
    \end{aligned}
\end{equation}
respectively. Here, $\mathrm{AC}$ denotes the class of absolutely continuous functions and $\overline{L}$ is the Legendre transform of $\overline{H}:\mathbb{R}^n \times \mathbb{R}^n \to \mathbb{R
}$.


\subsection{Main results and proof strategies}

To establish our main result, we first introduce the following notation that can be viewed as a metric between any two points $x$ and $y$ in $\mathbb{R}^n$.

\begin{definition}\label{def:notation}
Let $c,x,y\in\mathbb{R}^n,\ \epsilon>0,\ 0\leq t_1\leq t_2<+\infty$. Define
\begin{align*}
\Gamma(t_1,t_2,x,y)&:=\left\{\gamma\in\mathrm{AC}([t_1,t_2],\mathbb{R}^n):\ \gamma(t_1)=x,\ \gamma(t_2)=y\right\},\\
m^{\epsilon}(t_1,t_2,x,y)&:=\inf\left\{\int_{t_1}^{t_2}L\left(\gamma(s),\frac{\gamma(s)}{\epsilon},-\dot{\gamma}(s)\right)ds\ :\ \gamma\in\Gamma(t_1,t_2,x,y)\right\},\\
m^{\epsilon}_{c}(t_1,t_2,x,y)&:=\inf\left\{\int_{t_1}^{t_2}L\left(c,\frac{\gamma(s)}{\epsilon},-\dot{\gamma}(s)\right)ds\ :\ \gamma\in\Gamma(t_1,t_2,x,y)\right\},\\
\overline{m}(t_1,t_2,x,y)&:=\inf\left\{\int_{t_1}^{t_2}\overline{L}(\overline{\gamma}(s),-\dot{\overline{\gamma}}(s))ds\ :\ \overline{\gamma}\in\Gamma(t_1,t_2,x,y)\right\},\\
\overline{m}_{c}(t_1,t_2,x,y)&:=\inf\left\{\int_{t_1}^{t_2}\overline{L}(c,-\dot{\overline{\gamma}}(s))ds\ :\ \overline{\gamma}\in\Gamma(t_1,t_2,x,y)\right\}.
\end{align*}
\end{definition}

Although only the time difference $t_2-t_1$ impacts the calculation of the cost in the above notations, we still specify the start and end time points to maintain consistency with the notation used for the discounted static problem.

We note that the optimal control formulas \eqref{eqn:ocfue}, \eqref{eqn:ocfuebar} can be reformulated as
\begin{equation*}
    \begin{aligned}
        u^\epsilon (x, t)&= \inf \left\lbrace m^{\epsilon}(0,t,x,y)+g\left(y\right): y\in\mathbb{R}^n\right\rbrace
    \end{aligned}
\end{equation*}
and
\begin{equation*}
    \begin{aligned}
        u (x, t)&= \inf \left\lbrace \overline{m}(0,t,x,y)+g\left(y\right): y\in\mathbb{R}^n\right\rbrace,
    \end{aligned}
\end{equation*}
respectively.

We now present our main result, which establishes a rate of $O (\sqrt{\epsilon})$ for the multiscale setting.

\begin{theorem} \label{thm:main}
    Assume {\rm(H1)-(H4)} and let $g \in \mathrm{BUC}(\mathbb{R}^n) \cap \mathrm{Lip} (\mathbb{R}^n) $. For $\epsilon >0$, let $u^{\epsilon}$ be the unique viscosity solution to \eqref{eqn:ms} and $u$ be the unique viscosity solution to \eqref{eqn:eh}. Then there exists a constant $C>0$ depending only on $n$, $H$ and $\|Dg\|_{L^{\infty}(\mathbb{R}^n)}$ such that for any $(x,t)\in\mathbb{R}^n\times(0,\infty)$ and $\epsilon\in(0,1)$, we have
    \begin{equation}\label{conclusion}
    \begin{aligned}
        &|u^\epsilon(x, t)-u(x, t)| \leq C t\sqrt{\epsilon}, \qquad \qquad \text{ if } t \geq \sqrt{\epsilon},\\
        &|u^\epsilon(x, t)-u(x, t)| \leq C \min \left\{t, \epsilon \right\}, \quad \,\text{ if } 0<t < \sqrt{\epsilon}.
    \end{aligned}
    \end{equation}
\end{theorem}

We also state a similar result for the static problem.

\begin{theorem}\label{thm:main2}
Assume {\rm(H1)-(H4)}. For $\lambda,\epsilon>0$, let $u^{\epsilon}$ be the unique viscosity solution to
\begin{equation}\label{discountepsilon}
\lambda u^{\epsilon}+H\left(x,\frac{x}{\epsilon},Du^{\epsilon}\right)=0,
\end{equation}
and let $u$ be the unique viscosity solution to
\begin{equation}\label{discountbar}
\lambda u+\overline{H}\left(x,Du\right)=0.
\end{equation}
Then, there exists a constant $C>0$ depending only on $n$ and $H$ such that and $\lambda,\epsilon\in(0,1)$, we have
\begin{align}\label{conclusion2}
\|u^{\epsilon}-u\|_{L^{\infty}(\mathbb{R}^n)}\leq \frac{C\sqrt{\epsilon}}{\lambda}.
\end{align}
\end{theorem}

A notable difference (for Cauchy the problems) between the multiscale setting and the case where Hamiltonians $H=H(y, p)$ (as studied in \cite{TranYu2022}) is that the rate of convergence in the former depends on time $t$, as opposed to being uniform in $t$ for the latter. Specifically, in the multiscale setting, for $t$ large, the rate of convergence is $O(\sqrt{\epsilon})$, with the power of $\epsilon$ being $\frac{1}{2}$. This power arises from balancing the macroscale and microscale variables, which is a key feature of the multiscale setting.

We now outline the proof strategy for the lower bound when $t\geq \sqrt{\epsilon}$ in the multiscale setting. As the proof of Theorem \ref{thm:main2} is based on exactly the same idea, we focus on presenting the proof idea of Theorem \ref{thm:main}.

First, we consider a minimizing curve $\gamma_0:\left[0,t\right] \to \mathbb{R}^n$ for $u^\epsilon(x, t)$, i.e., $\displaystyle \gamma_0\left(0\right)=x$ and
\begin{equation} \label{eqn:demo}
u^\epsilon(x, t) = \int_0^t L\left(\gamma_0(s), \frac{\gamma_0(s)}{\epsilon}, -\dot{\gamma}_0(s)\right)ds+g(\gamma_0(t)).
\end{equation}
The main idea is to break $\gamma_0$ into $N$ evenly spaced pieces with respect to time, where $N$ needs to be determined appropriately. For each piece, we approximate its cost by fixing the first argument of $L$ in \eqref{eqn:demo}. More precisely, for the $k$-th piece where $k=0, 1, \cdots, N-1$, the time runs from $t_k=k\sqrt{\epsilon}$ to $t_{k+1}=(k+1)\sqrt{\epsilon}$, and we estimate the running cost within this time with the first argument fixed in Lagrangian by the value of the curve at the beginning $x_k=\gamma_0(t_k)$ of this piece, that is,
\begin{equation}\label{eqn:demofix}
\sum_{k=0}^{N-1}\int_{t_k}^{t_{k+1}}L\left(x_k,\frac{\gamma_0(s)}{\epsilon},-\dot{\gamma}_0(s)\right)ds
\end{equation}
The error for fixing the first argument of $L$ in the running cost for $N$ pieces of shorter curves is $\frac{t^2}{N}$ (under condition (H4), see Lemma \ref{lem:LbarLip}). 

For each piece with the first argument of $L$ fixed in the cost, we can use the definition of $m^{\epsilon}_{x_k}$ to obtain
\begin{equation}\label{eqn:demometric}
\sum_{k=0}^{N-1}\int_{t_k}^{t_{k+1}}L\left(x_k,\frac{\gamma_0(s)}{\epsilon},-\dot{\gamma}_0(s)\right)ds\geq\sum_{k=0}^{N-1}m^{\epsilon}_{x_k}(t_k,t_{k+1},\gamma_0(t_k),\gamma_0(t_{k+1})).
\end{equation}
Further, we can use the following lemma to connect $m^{\epsilon}_{x_k}$ with $\overline{m}_{x_k}$ and hence $u(x ,t)$.

\begin{lemma}\label{lem:metric}
Assume {\rm(H1)-(H3)}. Fix $c\in\mathbb{R}^n$. Let $x, y \in \mathbb{R}^n, \epsilon, t > 0$ and $M_0>0$ with $|y-x|\leq M_0t$. Let $K_0>0$ be a constant that satisfies \eqref{K_0H}, \eqref{K_0L}. Then, there exists a constant $C=C(n,M_0,K_0)>0$ such that
\begin{equation}\label{eqn:presult}
    \left|m^{\epsilon}_{c}(0,t,x,y)-\overline{m}_{c}(0,t,x,y)\right| \leq C \epsilon.
\end{equation}
\end{lemma}

\begin{remark}
This lemma is a generalization of \cite[Lemmas 3.1, 3.2]{TranYu2022}. In \cite{TranYu2022}, it is proved that for a fixed $c \in \mathbb{R}^n$, and for the Lagrangian $L^{c}(\cdot,\cdot)=L(c,\cdot,\cdot)$, there exists a constant $C=C(n,L^{c},M_0)>0$ such that for any $x, y \in \mathbb{R}^n,$ $\epsilon, t>0$ with $|y-x|\leq M_0t$, we have the conclusion of Lemma \ref{lem:metric} as above. Although the constant $C=C(n,L^{c},M_0)>0$ could potentially depend on ${c}\in\mathbb{R}^n$ due to the dependency of $L^{c}$ on ${c}\in\mathbb{R}^n$, it can be shown, under the assumptions {\rm(H1)-(H3)}, that the constant $C>0$ depends only on $n,M_0,K_0$, as presented in Appendix. 
\end{remark}

Using Lemma \ref{lem:metric}, we can approximate each term on the right-hand side of \eqref{eqn:demometric} by $\overline{m}_{x_k}$ with the corresponding arguments, incurring an error of $\epsilon N$ for the sum of $N$ terms. Furthermore, by constructing an admissible path for $u(x, t)$, we can replace $\overline{m}_{x_k}$ with $u(x, t)$, introducing an additional error of $\frac{t^2}{N}$. Thus, we obtain the inequality
\[
u^\epsilon(x, t) \geq u(x,t)-C\frac{t^2}{N}-CN\epsilon,
\]
where $C > 0$ is a constant that depends only on $n,M_0,K_0$. In summary, we have one source of error coming from fixing the macroscale variable in approximating the running cost and the other source of error caused by handling the microscale variable with Lemma \ref{lem:metric}. To minimize the total error, that is, to balance between $\frac{t^2}{N}$ and $N\epsilon$, the best $N$ we can choose is $N=\frac{t}{\sqrt{\epsilon}}$, which yields a bound of $Ct\sqrt{\epsilon}$ on the total error. 

The balance between the spatial variable and the oscillatory variable in homogenization is a key feature of the multiscale setting, and it is the first work in the literature where scale separations occur at the level of optimal curves for the solutions. As we can see, it is crucial in the proof of Theorem \ref{thm:main} that the constant $C>0$ in Lemma \ref{lem:metric} is independent of ${c}\in\mathbb{R}^n$, as we freeze the spatial variable at various places along minimizing curves. Also, we will see that the involvement of time $t$ in the bound $Ct\sqrt{\epsilon}$ is necessary by an example, which is also a feature distinguished from the case where Hamiltonians do not depend on the spatial variable.

\begin{remark}\label{rem:cond4}
Condition {\rm(H4)} is a necessary assumption for the approach of fixing the $x$-arguments to work. This condition enables us to bound the error caused by freezing the spatial variable. In Proposition \ref{exp:nolip}, we provide an illustration of the case where {\rm(H4)} is not satisfied, and the error cannot be controlled in this way.
\end{remark}

\subsection*{Organization of this paper}
In Section \ref{sec:provethm1}, we prove Theorem \ref{thm:main}. In Section \ref{sec:provethm2}, we verify Theorem \ref{thm:main2}. In Section \ref{sec:example}, we provide examples that demonstrate the optimality of the rate of convergence suggested in Theorem \ref{thm:main}. In Appendix, we show Lemma \ref{lem:metric} in detail.

\section{Proof of Theorem \ref{thm:main}}\label{sec:provethm1}

\subsection{Preliminaries}

We begin by stating that throughout this paper, we will use $C,C_0,K_0,M,M_0>0$ to denote positive constants, and their dependence on parameters will be specified as their arguments. The constants $C_0=C_0\left(H,\left\|Dg\right\|_{L^\infty(\mathbb{R}^n)}\right),\,K_0=K_0\left(H,\left\|Dg\right\|_{L^\infty(\mathbb{R}^n)}\right),$ $M_0=M_0\left(H, \, \left\|Dg\right\|_{L^\infty(\mathbb{R}^n)}\right)>0$ will be fixed throughout this paper, while $C,M>0$ may vary line by line.

Prior to proving Theorem \ref{thm:main}, we introduce two essential lemmas that will assist us in constraining the errors that arise when we freeze the first argument of $L$ in the running cost.

We first state the lemma about the boundedness of velocities of minimizing curves.

\begin{lemma}\label{lem:coercivity}
Assume {\rm(H1)-(H3)}. Let $x\in \mathbb{R}^n$, $t>0$ and $ \epsilon>0$. Suppose that $\gamma:\left[0, t\right] \to \mathbb{R}^n$ is a minimizing curve of $u^\epsilon(x, t)$ in the sense that $\gamma$ is absolutely continuous, and
\begin{equation}
u^\epsilon (x, t) = \int_0^t L\left(\gamma(s), \frac{\gamma(s)}{\epsilon} ,-\dot{\gamma}(s)\right) ds+g\left(\gamma(t)\right)  
\end{equation}
with $\gamma\left(0\right)=x$.
Then, there exists a constant $M_0=M_0\left(H, \|Dg\|_{L^\infty(\mathbb{R}^n)}\right)>0$ such that $ \left\|\dot{\gamma}\right\|_{L^\infty([0,t])} \leq M_0$.
Similarly, if $\overline{\gamma}:[0,t]\to\mathbb
{R}^n$ is a minimizing curve of $u(x, t)$ in the sense that $\overline{\gamma}$ is absolutely continuous, and
\begin{equation}
u (x, t) = \int_0^t \overline{L} \left( \overline{\gamma}(s), -\dot{\overline{\gamma}}(s)\right) ds+g\left(\overline{\gamma}(t)\right)
\end{equation}
with $\overline{\gamma}(0)=x$, then there exists a constant $M_0=M_0\left(H, \|Dg\|_{L^\infty(\mathbb{R}^n)}\right)>0$ such that $ \left\|\dot{\overline{\gamma}}\right\|_{L^\infty([0,t])} \leq M_0$.
\end{lemma}

The following lemma states that $L(\cdot, y, v)$ and $\overline{L}(\cdot, v)$ are Lipschitz uniformly in $y$ and $v$ under the condition (H4). 

\begin{lemma}\label{lem:LbarLip}
Assume {\rm(H1)-(H4)}. Then,
\[
\left|L(x_1, y, v)-L(x_2, y, v)\right| \leq \mathrm{Lip}(H) |x_1-x_2|,
\]
and
\[
\left|\overline{L}(x_1, v)-\overline{L}(x_2, v)\right| \leq \mathrm{Lip}(H) |x_1-x_2|,
\]
for any $x_1, x_2 \in \mathbb{R}^n$, $y \in \mathbb{T}^n$, and $v \in \mathbb{R}^n$.
\end{lemma}

The proofs of the above two lemmas are omitted here. See \cite{tran_hamilton-jacobi_2021} for more details.

\subsection{Proof}
We are now ready to prove Theorem \ref{thm:main}.

\begin{proof}[Proof of Theorem \ref{thm:main}]
Let $x\in\mathbb{R}^n,\ \epsilon,t>0$. We first show that for some constant $C=C(n,H,\|Dg\|_{L^{\infty}(\mathbb{R}^n)})>0$, it holds that $u^{\epsilon}(x,t)-u(x,t)\geq-Ct\sqrt{\epsilon}$ for $t\geq\sqrt{\epsilon}$, and that $u^{\epsilon}(x,t)-u(x,t)\geq-C\epsilon$ for $t\in(0,\sqrt{\epsilon})$.

\medskip

Let $\gamma_0:[0,t]\to\mathbb{R}^n$ be an absolutely continuous curve with $\gamma_0(0)=x$ such that
$$
u^{\epsilon}(x,t)=\int_0^tL\left(\gamma_0(s),\frac{\gamma_0(s)}{\epsilon},-\dot{\gamma_0}(s)\right)ds+g(\gamma_0(t))=m^{\epsilon}(0,t,x,y)+g(y),
$$
where $y$ denotes the point $\gamma_0(t)\in\mathbb{R}^n$. Then,
$$
u(x,t)\leq\overline{m}(0,t,x,y)+g(y),
$$
and thus,
\begin{align}\label{lowerboundinm}
u^{\epsilon}(x,t)-u(x,t)\geq m^{\epsilon}(0,t,x,y)-\overline{m}(0,t,x,y).
\end{align}

\medskip

In order to give a lower bound of $m^{\epsilon}(0,t,x,y)-\overline{m}(0,t,x,y)$, we consider a partition
$$
0=t_0\leq t_1\leq t_2\leq \cdots \leq t_k\leq t_{k+1}\leq\cdots \leq t_N\leq t_{N+1}=t
$$
of the interval $[0,t]$, where $N$ is a nonnegative integer that will be determined later together with the division. On each interval $[t_k,t_{k+1}]$ for $k=0,\cdots,N$, we freeze the spatial variable, homogenize in the oscillatory variable, and then unfreeze the spatial variable in divided steps as follows. We finally estimate the commutators arising from these steps.

\medskip

Step 1: Freeze the spatial variable.

\medskip

For each $k=0,\cdots,N$, let $x_k:=\gamma_0(t_k)$. Then, for each $k=0,\cdots,N,$
$$
m^{\epsilon}(t_k,t_{k+1},\gamma_0(t_k),\gamma_{0}(t_{k+1}))=\int_{t_k}^{t_{k+1}}L\left(x_k,\frac{\gamma_0(s)}{\epsilon},-\dot{\gamma_0}(s)\right)ds+E_k,
$$
where
$$
E_k:=\int_{t_k}^{t_{k+1}}L\left(\gamma_0(s),\frac{\gamma_0(s)}{\epsilon},-\dot{\gamma_0}(s)\right)ds-\int_{t_k}^{t_{k+1}}L\left(x_k,\frac{\gamma_0(s)}{\epsilon},-\dot{\gamma_0}(s)\right)ds.
$$

\medskip

Step 2: Homogenize in the oscillatory variable.

\medskip

We apply Lemma \ref{lem:metric} and Lemma \ref{lem:coercivity} to see that there exists a constant $C=C(n,H,\|Dg\|_{L^{\infty}(\mathbb{R}^n)})>0$ such that
\begin{align*}
\int_{t_k}^{t_{k+1}}L\left(x_k,\frac{\gamma_0(s)}{\epsilon},-\dot{\gamma_0}(s)\right)ds&\geq m^{\epsilon}_{x_k}(t_k,t_{k+1},\gamma_0(t_k),\gamma_0(t_{k+1}))\\
&\geq\overline{m}_{x_k}(t_k,t_{k+1},\gamma_0(t_k),\gamma_0(t_{k+1}))-C\epsilon
\end{align*}
for each $k=0,\cdots,N$. It is a crucial fact that the constant $C>0$ is independent of $k=0,\cdots,N,$ i.e., independent of the spatial positions.

\medskip

Step 3: Unfreeze the spatial variable.

\medskip

For each $k=0,\cdots,N$, let $\overline{\gamma}_k:[t_k,t_{k+1}]\to\mathbb{R}^n$ be an absolutely continuous curve with $\overline{\gamma}_k(t_k)=\gamma_0(t_k),\ \overline{\gamma}_{k}(t_{k+1})=\gamma_0(t_{k+1})$ such that
\begin{align*}
\overline{m}_{x_k}(t_k,t_{k+1},\gamma_0(t_k),\gamma_0(t_{k+1}))=\int_{t_k}^{t_{k+1}}\overline{L}(x_k,-\dot{\overline{\gamma}}_k(s))ds.
\end{align*}
Then,
\begin{align*}
\overline{m}_{x_k}(t_k,t_{k+1},\gamma_0(t_k),\gamma_0(t_{k+1}))&=\int_{t_k}^{t_{k+1}}\overline{L}(\overline{\gamma}_k(s),-\dot{\overline{\gamma}}_k(s))ds-\overline{E}_k\\
&\geq\overline{m}(t_k,t_{k+1},\gamma_0(t_k),\gamma_0(t_{k+1}))-\overline{E}_k,
\end{align*}
where
\begin{align*}
\overline{E}_k:=\int_{t_k}^{t_{k+1}}\overline{L}(\overline{\gamma}_k(s),-\dot{\overline{\gamma}}_k(s))ds-\int_{t_k}^{t_{k+1}}\overline{L}(x_k,-\dot{\overline{\gamma}}_k(s))ds
\end{align*}
for each $k=0,\cdots,N$.

\medskip

Step 4: Estimate the errors $E_k,\overline{E}_k$ and obtain a lower bound.

\medskip

From Steps 1-3, we have that for each $k=0,\cdots,N$,
\begin{align*}
m^{\epsilon}(t_k,t_{k+1},\gamma_0(t_k),\gamma_0(t_{k+1}))\geq \overline{m}(t_k,t_{k+1},\gamma_0(t_k),\gamma_0(t_{k+1}))+E_k-\overline{E}_k-C\epsilon.
\end{align*}
Since
\begin{align*}
m^{\epsilon}(0,t,x,y)=\sum_{k=0}^Nm^{\epsilon}(t_k,t_{k+1},\gamma_0(t_k),\gamma_0(t_{k+1}))
\end{align*}
and
\begin{align*}
\overline{m}(0,t,x,y)\leq\sum_{k=0}^N\overline{m}(t_k,t_{k+1},\gamma_0(t_k),\gamma_0(t_{k+1})),
\end{align*}
we obtain
\begin{align}
m^{\epsilon}(0,t,x,y)\geq\overline{m}(0,t,x,y)+\sum_{k=0}^N\left(E_k-\overline{E}_k-C\epsilon\right).\label{metriclowerbound}
\end{align}

Now, we estimate the errors $E_k,\overline{E}_k$. By Lemmas \ref{lem:coercivity}, \ref{lem:LbarLip}, we get
\begin{align*}
|E_k|&=\left|\int_{t_k}^{t_{k+1}}L\left(\gamma_0(s),\frac{\gamma_0(s)}{\epsilon},-\dot{\gamma_0}(s)\right)ds-\int_{t_k}^{t_{k+1}}L\left(x_k,\frac{\gamma_0(s)}{\epsilon},-\dot{\gamma_0}(s)\right)ds\right|\\
&\leq\int_{t_k}^{t_{k+1}}\mathrm{Lip}(H)|\gamma_0(s)-x_k|ds\\
&\leq\mathrm{Lip}(H)M_0\int_{t_k}^{t_{k+1}}|s-t_k|ds\\
&\leq\mathrm{Lip}(H)M_0(t_{k+1}-t_k)^2
\end{align*}
for each $k=0,\cdots,N$. By the same estimate, we also get $|\overline{E}_k|\leq\mathrm{Lip}(H)M_0(t_{k+1}-t_k)^2$ for each $k=0,\cdots,N$.

Set $N=\left\lfloor\frac{t}{\sqrt{\epsilon}}\right\rfloor$ and $t_k=k\sqrt{\epsilon}$ for each $k=0,\cdots,N$. With this choice of division, it holds that $t_{k+1}-t_k\leq\sqrt{\epsilon}$ for all $k=0,\cdots,N$. Note that $t_N=t_{N+1}=t$ when $\frac{t}{\sqrt{\epsilon}}$ is a positive integer. If $t\in(0,\sqrt{\epsilon})$, then $N=0$, and thus,
\begin{align*}
\left|\sum_{k=0}^N\left(E_k-\overline{E}_k-C\epsilon\right)\right|\leq 2\mathrm{Lip}(H)M_0\epsilon+C\epsilon
\leq C\epsilon
\end{align*}
with $C=C(n,H,\|Dg\|_{L^{\infty}(\mathbb{R}^n)})>0$ changed to a larger constant in the last inequality. If $t\geq\sqrt{\epsilon}$, then $N+1\leq\frac{2t}{\sqrt{\epsilon}}$, and thus,
\begin{align*}
\left|\sum_{k=0}^N\left(E_k-\overline{E}_k-C\epsilon\right)\right|&\leq 2(N+1)\mathrm{Lip}(H)M_0\epsilon+(N+1)C\epsilon\\
&\leq 4\mathrm{Lip}(H)M_0t\sqrt{\epsilon}+2Ct\sqrt{\epsilon}\\
&\leq Ct\sqrt{\epsilon}
\end{align*}
with $C=C(n,H,\|Dg\|_{L^{\infty}(\mathbb{R}^n)})>0$ changed to a larger constant in the last inequality. In all cases, we obtain a desired lower bound by combining \eqref{lowerboundinm}, \eqref{metriclowerbound}.

\medskip

To prove an upper bound of $u^{\epsilon}(x,t)-u(x,t)$, we instead obtain a lower bound of $u(x,t)-u^{\epsilon}(x,t)$. Since Lemmas \ref{lem:metric}, \ref{lem:coercivity}, \ref{lem:LbarLip} are written entirely symmetric in $m^{\epsilon}$ and $\overline{m}$, $L$ and $\overline{L}$, the same arguments as the above (but swapping $u^{\epsilon}$ and $u$, $m^{\epsilon}$ and $\overline{m}$, $L$ and $\overline{L}$, respectively) also prove lower bounds $u(x,t)-u^{\epsilon}(x,t)\geq-Ct\sqrt{\epsilon}$ for $t\geq\sqrt{\epsilon}$ and $u(x,t)-u^{\epsilon}(x,t)\geq-C\sqrt{\epsilon}$ for $t\in(0,\sqrt{\epsilon})$.

\medskip

Finally, for $t\in(0,\sqrt{\epsilon})$, we apply the comparison principle to see that there exists a constant $C=C(H,\|Dg\|_{L^{\infty}(\mathbb{R}^n)})>0$ such that
\begin{align*}
|u^{\epsilon}(x,t)-g(x)|\leq Ct
\end{align*}
and
\begin{align*}
|u(x,t)-g(x)|\leq Ct.
\end{align*}
Therefore, there exists a constant $C=C(H,\|Dg\|_{L^{\infty}(\mathbb{R}^n)})>0$ such that
\begin{align*}
|u^{\epsilon}(x,t)-u(x,t)|\leq Ct,
\end{align*}
which yields \eqref{conclusion} together with the bounds $|u^{\epsilon}(x,t)-u(x,t)|\leq Ct\sqrt{\epsilon}$ for $t\geq\sqrt{\epsilon}$ and $|u^{\epsilon}(x,t)-u(x,t)|\leq C\epsilon$ for $t\in(0,\sqrt{\epsilon})$. This completes the proof.
\end{proof}

\section{Proof of Theorem \ref{thm:main2}}\label{sec:provethm2}
\subsection{Preliminaries} Let $u^{\epsilon}$ be the unique viscosity solution to \eqref{discountepsilon}, and let $u$ be the unique viscosity solution to \eqref{discountbar}. Then, we have the optimal control formulas for $u^{\epsilon}$ and $u$, that is,
\begin{equation}\label{eqn:ocfuestatic}
    \begin{aligned}
        u^\epsilon (x) = \inf \left\lbrace \int_0^{\infty} e^{-\lambda s} L\left(\gamma(s), \frac{\gamma(s)}{\epsilon} ,-\dot{\gamma}(s)\right) ds:\ \gamma(0)=x,\ \gamma\in \mathrm{AC}([0,T];\mathbb{R}^n),\ \textrm{for any $T>0$}\right\rbrace
    \end{aligned}
\end{equation}
and
\begin{equation}\label{eqn:ocfuebarstatic}
    \begin{aligned}
        u(x) = \inf \left\lbrace \int_0^{\infty} e^{-\lambda s} \overline{L} \left( \overline{\gamma}(s), -\dot{\overline{\gamma}}(s)\right) ds:\ \overline{\gamma}(0)=x,\ \overline{\gamma}\in \mathrm{AC}([0,T];\mathbb{R}^n),\ \textrm{for any $T>0$}\right\rbrace,
    \end{aligned}
\end{equation}
respectively.

We state the lemma about the boundedness of velocities of minizing curves for this problem, which corresponds to Lemma \ref{lem:coercivity}.

\begin{lemma}\label{lem:coercivitystatic}
Assume {\rm(H1)-(H3)}. Let $x\in \mathbb{R}^n$, and $\lambda,\epsilon>0$. Suppose that $\gamma:[0, \infty) \to \mathbb{R}^n$ is a minimizing curve of $u^\epsilon(x)$ in the sense that $\gamma\in \mathrm{AC}([0,T];\mathbb{R}^n)$ for any $T>0$, and
\begin{equation}
u^\epsilon (x) = \int_0^{\infty} e^{-\lambda s} L\left(\gamma(s), \frac{\gamma(s)}{\epsilon} ,-\dot{\gamma}(s)\right) ds
\end{equation}
with $\gamma\left(0\right)=x$.
Then, there exists a constant $M_0=M_0\left(H\right)>0$ such that $ \left\|\dot{\gamma}\right\|_{L^\infty([0,\infty))} \leq M_0$.
Similarly, if $\overline{\gamma}:[0,\infty)\to\mathbb
{R}^n$ is a minimizing curve of $u(x)$ in the sense that $\overline{\gamma}\in \mathrm{AC}([0,T];\mathbb{R}^n)$ for any $T>0$, and
\begin{equation}
u (x) = \int_0^{\infty} e^{-\lambda s} \overline{L} \left( \overline{\gamma}(s), -\dot{\overline{\gamma}}(s)\right) ds
\end{equation}
with $\overline{\gamma}(0)=x$, then there exists a constant $M_0=M_0\left(H\right)>0$ such that $ \left\|\dot{\overline{\gamma}}\right\|_{L^\infty([0,\infty))} \leq M_0$.
\end{lemma}

Also, by applying the comparison principle, we have the $L^{\infty}$-bound of $u^{\epsilon}$ and $u$.

\begin{lemma}\label{lem:Linfinitystatic}
Assume {\rm(H1)-(H3)}. Let $u^{\epsilon}$ be the unique viscosity solution to \eqref{discountepsilon}, and let $u$ be the unique viscosity solution to \eqref{discountbar}. Let $M:=\|H(\cdot,\cdot,0)\|_{L^{\infty}(\mathbb{R}^n\times\mathbb{T}^n)}$. Then,
$$
\|u^{\epsilon}\|_{L^{\infty}(\mathbb{R}^n)}\leq\frac{M}{\lambda}
$$
and
$$
\|u\|_{L^{\infty}(\mathbb{R}^n)}\leq\frac{M}{\lambda}.
$$
\end{lemma}

\subsection{Proof} We introduce the additional notations with the discount term for the proof of Theorem \ref{thm:main2}; for $x,y\in\mathbb{R}^n$, $\lambda,\epsilon>0$, $0\leq t_1\leq t_2<+\infty$, we let
\begin{align*}
m^{\epsilon}_{\lambda}(t_1,t_2,x,y)&:=\inf\left\{\int_{t_1}^{t_2}e^{-\lambda s}L\left(\gamma(s),\frac{\gamma(s)}{\epsilon},-\dot{\gamma}(s)\right)ds\ :\ \gamma\in\Gamma(t_1,t_2,x,y)\right\},\\
\overline{m}_{\lambda}(t_1,t_2,x,y)&:=\inf\left\{\int_{t_1}^{t_2}e^{-\lambda s}\overline{L}(\overline{\gamma}(s),-\dot{\overline{\gamma}}(s))ds\ :\ \overline{\gamma}\in\Gamma(t_1,t_2,x,y)\right\}.
\end{align*}
Now we prove Theorem \ref{thm:main2}.

\begin{proof}[Proof of Theorem \ref{thm:main2}]
Let $M:=\|H(\cdot,\cdot,0)\|_{L^{\infty}(\mathbb{R}^n\times\mathbb{T}^n)}$. Let $H^M(x,y,p):=H(x,y,p)-M$ and $\overline{H}^M$ be its effective Hamiltonian, which coincides with $\overline{H}-M$. Then, $u^{\epsilon}_M:=u^{\epsilon}+\frac{M}{\lambda}$ ($u_M:=u+\frac{M}{\lambda}$, resp.) is the unique viscosity solution to
$$
\lambda u^{\epsilon}_M+H^M\left(x,\frac{x}{\epsilon},Du^{\epsilon}_M\right)=0\ \ \ \ \ \ \left(\lambda u_M+\overline{H}^M\left(x,Du_M\right)=0,\ \textrm{resp.}\right).
$$
The additional property of the Hamiltonian $H^M$ is that its Lagrangian and effective Lagrangian are nonnegative. Since $u^{\epsilon}_M-u_M=u^{\epsilon}-u$, it suffices to prove that $\|u^{\epsilon}_M-u_M\|_{L^{\infty}(\mathbb{R}^n)}\leq \frac{C\sqrt{\epsilon}}{\lambda}.$ Therefore, it suffices to prove that $\|u^{\epsilon}-u\|_{L^{\infty}(\mathbb{R}^n)}\leq \frac{C\sqrt{\epsilon}}{\lambda}$ when $L,\overline{L}\geq0$, which we assume from now on without loss of generality.

\medskip

Let $x\in\mathbb{R}^n$, and let $\lambda,\epsilon\in(0,1)$. The goal is to prove $u^{\epsilon}(x)-u(x)\geq-\frac{C\sqrt{\epsilon}}{\lambda}$ for some constant $C=C(n,H)>0$. Let $\gamma_0\in\mathrm{AC}([0,+\infty);\mathbb{R}^n)$ be a curve such that with $\gamma_0(0)=x$, $\gamma_0\in\mathrm{AC}([0,T];\mathbb{R}^n)$ for any $T>0$, and
\begin{equation*}
u^\epsilon (x) = \int_0^{\infty} e^{-\lambda s} L\left(\gamma_0(s), \frac{\gamma_0(s)}{\epsilon} ,-\dot{\gamma_0}(s)\right) ds.
\end{equation*}
Consider a partition
$$
0=t_0< t_1< t_2< \cdots< t_k< t_{k+1}<\cdots
$$
of the interval $[0,+\infty)$ with $t_k\to+\infty$ as $k\to+\infty$, which will be determined later.

\medskip

Step 1: Freeze the spatial variable.

\medskip

For each $k=0,1,2,\cdots,$ let $x_k:=\gamma_0(t_k)$. Then, for each $k=0,1,2,\cdots,$
$$
m^{\epsilon}_{\lambda}(t_k,t_{k+1},\gamma_0(t_k),\gamma_{0}(t_{k+1}))=\int_{t_k}^{t_{k+1}}e^{-\lambda s}L\left(x_k,\frac{\gamma_0(s)}{\epsilon},-\dot{\gamma_0}(s)\right)ds+E_k,
$$
where
$$
E_k:=\int_{t_k}^{t_{k+1}}e^{-\lambda s}L\left(\gamma_0(s),\frac{\gamma_0(s)}{\epsilon},-\dot{\gamma_0}(s)\right)ds-\int_{t_k}^{t_{k+1}}e^{-\lambda s}L\left(x_k,\frac{\gamma_0(s)}{\epsilon},-\dot{\gamma_0}(s)\right)ds.
$$

\medskip

Step 2: Homogenize in the oscillatory variable.

\medskip

We apply Lemma \ref{lem:metric} and Lemma \ref{lem:coercivitystatic} to see that there exists a constant $C=C(n,H)>0$ such that
\begin{align*}
\int_{t_k}^{t_{k+1}}e^{-\lambda s}L\left(x_k,\frac{\gamma_0(s)}{\epsilon},-\dot{\gamma_0}(s)\right)ds&\geq e^{-\lambda t_{k+1}}\int_{t_k}^{t_{k+1}}L\left(x_k,\frac{\gamma_0(s)}{\epsilon},-\dot{\gamma_0}(s)\right)ds\\
&\geq e^{-\lambda t_{k+1}}m^{\epsilon}_{x_k}(t_k,t_{k+1},\gamma_0(t_k),\gamma_0(t_{k+1}))\\
&\geq e^{-\lambda t_{k+1}}\overline{m}_{x_k}(t_k,t_{k+1},\gamma_0(t_k),\gamma_0(t_{k+1}))-Ce^{-\lambda t_{k+1}}\epsilon.
\end{align*}
for each $k=0,1,2,\cdots$.

\medskip

Step 3: Unfreeze the spatial variable.

\medskip

For each $k=0,1,2,\cdots$, let $\overline{\gamma}_k:[t_k,t_{k+1}]\to\mathbb{R}^n$ be an absolutely continuous curve with $\overline{\gamma}_k(t_k)=\gamma_0(t_k),\ \overline{\gamma}_{k}(t_{k+1})=\gamma_0(t_{k+1})$ such that
\begin{align*}
\overline{m}_{x_k}(t_k,t_{k+1},\gamma_0(t_k),\gamma_0(t_{k+1}))=\int_{t_k}^{t_{k+1}}\overline{L}(x_k,-\dot{\overline{\gamma}}_k(s))ds.
\end{align*}
Then,
\begin{align*}
e^{-\lambda t_{k+1}}\overline{m}_{x_k}(t_k,t_{k+1},\gamma_0(t_k),\gamma_0(t_{k+1}))&=e^{-\lambda t_{k+1}}\int_{t_k}^{t_{k+1}}\overline{L}(\overline{\gamma}_k(s),-\dot{\overline{\gamma}}_k(s))ds-e^{-\lambda t_{k+1}}\overline{E}_k\\
&\geq e^{-\lambda (t_{k+1}-t_k)}\overline{m}_{\lambda}(t_k,t_{k+1},\gamma_0(t_k),\gamma_0(t_{k+1}))-e^{-\lambda t_{k+1}}\overline{E}_k,
\end{align*}
where
\begin{align*}
\overline{E}_k:=\int_{t_k}^{t_{k+1}}\overline{L}(\overline{\gamma}_k(s),-\dot{\overline{\gamma}}_k(s))ds-\int_{t_k}^{t_{k+1}}\overline{L}(x_k,-\dot{\overline{\gamma}}_k(s))ds
\end{align*}
for each $k=0,1,2,\cdots$.

\medskip

Step 4: Estimate the errors $E_k,\overline{E}_k$ and obtain a lower bound.

\medskip

From Steps 1-3, we have that for each $k=0,1,2,\cdots$,
\begin{align*}
m^{\epsilon}_{\lambda}(t_k,t_{k+1},\gamma_0(t_k),\gamma_0(t_{k+1}))\geq e^{-\lambda (t_{k+1}-t_k)}\overline{m}_{\lambda}(t_k,t_{k+1},\gamma_0(t_k),\gamma_0(t_{k+1}))+E_k-e^{-\lambda t_{k+1}}\overline{E}_k-Ce^{-\lambda t_{k+1}}\epsilon.
\end{align*}
Since
$$
u^{\epsilon}(x)=\sum_{k=0}^{\infty}m^{\epsilon}_{\lambda}(t_k,t_{k+1},\gamma_0(t_k),\gamma_0(t_{k+1}))
$$
and
$$
u(x)\leq\sum_{k=0}^{\infty}\overline{m}_{\lambda}(t_k,t_{k+1},\gamma_0(t_k),\gamma_0(t_{k+1})),
$$
we obtain
\begin{align}
u^{\epsilon}(x)\geq e^{-\lambda\sup_{k\geq0}(t_{k+1}-t_k)}u(x)+\sum_{k=0}^{\infty}\left(E_k-e^{-\lambda t_{k+1}}\overline{E}_k-Ce^{-\lambda t_{k+1}}\epsilon\right).\label{metriclowerboundstatic}
\end{align}

We now estimate the errors $E_k,\overline{E}_k$. By Lemmas \ref{lem:LbarLip}, \ref{lem:coercivitystatic}, we get
\begin{align*}
|E_k|&=\left|\int_{t_k}^{t_{k+1}}e^{-\lambda s}L\left(\gamma_0(s),\frac{\gamma_0(s)}{\epsilon},-\dot{\gamma_0}(s)\right)ds-\int_{t_k}^{t_{k+1}}e^{-\lambda s}L\left(x_k,\frac{\gamma_0(s)}{\epsilon},-\dot{\gamma_0}(s)\right)ds\right|\\
&\leq e^{-\lambda t_k}\int_{t_k}^{t_{k+1}}\mathrm{Lip}(H)|\gamma_0(s)-x_k|ds\\
&\leq e^{-\lambda t_k}\mathrm{Lip}(H)M_0(t_{k+1}-t_k)^2
\end{align*}
for each $k=0,1,2,\cdots$. Similarly, we also get $|\overline{E}_k|\leq\mathrm{Lip}(H)M_0(t_{k+1}-t_k)^2$ for each $k=0,1,2,\cdots$.

Set $t_k=k\sqrt{\epsilon}$ for each $k=0,1,2,\cdots$. Then, from \eqref{metriclowerboundstatic}, we have
\begin{align*}
u^{\epsilon}(x)\geq e^{-\lambda\sqrt{\epsilon}}u(x)-2\mathrm{Lip}(H)M_0\frac{\sqrt{\epsilon}}{\lambda}\sum_{k=0}^{\infty}\lambda \sqrt{\epsilon}e^{-k\lambda\sqrt{\epsilon}}-C\frac{\sqrt{\epsilon}}{\lambda}\sum_{k=0}^{\infty}\lambda \sqrt{\epsilon}e^{-k\lambda\sqrt{\epsilon}}.
\end{align*}

Also, by Lemma \ref{lem:Linfinitystatic},
$$
e^{-\lambda\sqrt{\epsilon}}u(x)-u(x)=-\left(1-e^{-\lambda\sqrt{\epsilon}}\right)u(x)\geq-\lambda\sqrt{\epsilon}u(x)\geq-M\sqrt{\epsilon},
$$
where $M:=\|H(\cdot,\cdot,0)\|_{L^{\infty}(\mathbb{R}^n\times\mathbb{T}^n)}$. By the elementary fact that $\sum_{k=0}^{\infty}\lambda \sqrt{\epsilon}e^{-k\lambda\sqrt{\epsilon}}$ is bounded by a universal constant for $\lambda,\epsilon\in(0,1)$, we have
$$
u^{\epsilon}(x)\geq u(x)-\frac{C\sqrt{\epsilon}}{\lambda}
$$
for some constant $C=C(n,H)>0$, as desired.

\medskip

To prove an upper bound of $u^{\epsilon}(x)-u(x)$, we instead obtain a lower bound of $u(x)-u^{\epsilon}(x)$ by interchanging $u^{\epsilon}$ and $u$, $m^{\epsilon}$ and $\overline{m}$, $L$ and $\overline{L}$, respectively, in the above arguments, which yields $u(x)\geq u^{\epsilon}(x)-\frac{C\sqrt{\epsilon}}{\lambda}$.
\end{proof}

\section{Examples}\label{sec:example}
In this section, we first establish the optimality of \eqref{conclusion} in Theorem \ref{thm:main}. We then provide an example to illustrate the necessity of condition (H4) for Theorem \ref{thm:main}. Finally, we present an example that demonstrates the necessity of involving the time variable $t$ in \eqref{conclusion}.

The following proposition demonstrates the optimality of the bound \eqref{conclusion}. We consider two cases: when $0<t < \sqrt{\epsilon}$ and $t\geq \sqrt{\epsilon}$. For $0 < t <\sqrt{\epsilon}$, the optimality of the rate of convergence in \eqref{conclusion} is evident from \eqref{lowerbound1}. Moreover, \eqref{lowerbound1} implies that the rate of convergence in \eqref{conclusion} is also optimal for $t=C \sqrt{\epsilon}$, where $C>1$, and in particular, for $t \geq \sqrt{\epsilon}$. This example is discussed in \cite[Proposition 4.3]{MitakeTranYu2019}, and it shows that the rate $O(\epsilon)$ is optimal when the Hamiltonian $H=H(y,p)$ depends solely on the oscillatory variable and momentum.

\begin{proposition}\label{prop:opst}
Consider the case where $n=1$, $H(y,p)=-V(y)+\frac{1}{2}p^2$ for a given continuous function $V\in C(\mathbb{T})$ with $\min_{\mathbb{T}}V=0$ and $V \geq 1$ on $[-3^{-1},3^{-1}]$, and $g \equiv 0$. For $\epsilon>0$, let $u^{\epsilon}$ be the solution to \eqref{eqn:ms}, and let $u$ be the solution to \eqref{eqn:eh}. Then, for $\epsilon\in(0,1)$ and $t>0$,
\begin{equation}\label{lowerbound1}
u^{\epsilon}(0,t)-u(0,t)\geq \frac{\sqrt{2}}{3}\min\left\{t,\epsilon\right\}.
\end{equation}
\end{proposition}
\begin{proof}
Due to the optimal control formula, and we have
$$
u^{\epsilon}(0,t)=\inf\left\{\epsilon\int_0^{\frac{t}{\epsilon}}V(\eta(s))+\frac{1}{2}|\dot{\eta}(s)|^2ds\ :\ \eta\in\mathrm{AC}\left(\left[0,\frac{t}{\epsilon}\right]\right),\ \eta(0)=0\right\}.
$$

Let $\eta\in\mathrm{AC}\left(\left[0,\frac{t}{\epsilon}\right]\right)$ with $\eta(0)=0$. If $\eta\left(\left[0,\frac{t}{\epsilon}\right]\right)\subset\left[-\frac{1}{3},\frac{1}{3}\right]$, then
$$
\epsilon\int_0^{\frac{t}{\epsilon}}V(\eta(s))+\frac{1}{2}|\dot{\eta}(s)|^2ds\geq\epsilon\int_0^{\frac{t}{\epsilon}}V(\eta(s))ds\geq t.
$$

If not, without loss of generality, we may assume that there exists $s_1\in\left(0,\frac{t}{\epsilon}\right)$ such that $\eta(s_1)=\frac{1}{3}$ and that $\eta([0,s_1))\subset\left(-\frac{1}{3},\frac{1}{3}\right)$. Then,
\begin{align*}
\epsilon\int_0^{\frac{t}{\epsilon}}V(\eta(s))+\frac{1}{2}|\dot{\eta}(s)|^2ds&\geq\epsilon\left(\int_0^{s_1}V(\eta(s))ds+\frac{1}{2}\int_0^{s_1}|\dot{\eta}(s)|^2ds\right)\\
&\geq\epsilon\left(s_1+\frac{1}{2s_1}\left|\int_0^{s_1}\dot{\eta}(s)ds\right|^2\right)=\epsilon\left(s_1+\frac{1}{18s_1}\right)\geq\frac{\sqrt{2}}{3}\epsilon.
\end{align*}
Since $u^{\epsilon}$ converges to $u\equiv0$ locally uniformly on $\mathbb{R}\times[0,\infty)$, we obtain \eqref{lowerbound1}.
\end{proof}

The following example explains why the assumption (H4) is needed.
\begin{proposition}\label{exp:nolip}
 Consider the Hamiltonian $H : \mathbb{R} \times \mathbb{T} \times \mathbb{R} \to \mathbb{R}$ defined by
\[
H(x,y,p):=-f(x)-W(y)+\frac{|p|^2}{2}
\]
where $f: \mathbb{R} \to \mathbb{R}$ is defined by
\[
f(x):= \left\{
\begin{aligned}
&|x|^\frac{1}{4}, \quad \text{if } |x| \leq 1,\\
& 1, \qquad \, \text{if } |x| > 1,
\end{aligned}
\right.\] and
$W : \mathbb{T} \to \mathbb{R}$ is defined by
\[
W(y)=\frac{1}{2}-|y|, \quad \text{ if } y \in \left[-\frac{1}{2}, \frac{1}{2}\right].
\]
 
Then, for $\epsilon\in(0,2^{-80})$, the corresponding solutions $u^\epsilon$ to \eqref{eqn:ms} and $u$ to \eqref{eqn:eh} with $g \equiv 0$ satisfy  
\[
u^\epsilon \left(0, 1\right)-u \left(0, 1\right) \geq \frac{1}{32}\epsilon^\frac{1}{4}.
\]
\end{proposition}

\begin{proof}
Let $H_1:\mathbb{T} \times \mathbb{R} \to \mathbb{R}$ be defined by $H_1(y,p):=-W(y)+\frac{|p|^2}{2}$. Then, the effective Hamiltonian $\overline{H}_1$ of $H_1$ is 
\[
\overline{H}_1(p):=
\left\{ \begin{aligned}
&0, \qquad \text{ if } |p| \leq \frac{2}{3}\\
& \lambda, \qquad \text{ if } |p| \geq \frac{2}{3}, \text{ where } \lambda >0 \text{ is a solution of } 2\sqrt{2} \int_0^\frac{1}{2} \sqrt{ \lambda + \frac{1}{2} - y} \, dy =|p|.
\end{aligned}
\right.
\]
In particular, for the Legendre transform of $\overline{H}_1$, denoted by $\overline{L}_1$, we know $\overline{L}_1(0)=0$ (see \cite{LionsPapaVara1987}, \cite{tran_hamilton-jacobi_2021}). 
The effective Hamiltonian $\overline{H}$ of $H$ is 
\[\overline{H}(x, p)= -f(x)+\overline{H}_1(p).\]
Hence, the optimal control formula of $u$ is
\begin{equation*}
    \begin{aligned}
        u(x, t) = \inf \left\lbrace \int_0^t f\left(\eta(s)\right) +\overline{L}_1\left(\dot{\eta}(s)\right) ds : \eta\in \mathrm{AC}([0,t];\mathbb{R}),\ \eta(t) = x\right\rbrace,
    \end{aligned}
\end{equation*}
which implies $u(0,1)=0$. 

Let $\gamma: \left[0, \frac{1}{\epsilon}\right] \to \mathbb{R}$ be a minimizing curve of $u^\epsilon(0,1)$ such that
\[u^\epsilon(0,1)=\epsilon\int_0^\frac{1}{\epsilon} \left(f\left(\epsilon\gamma(s)\right) + W(\gamma(s))+\frac{\left|\dot{\gamma}(s)\right|^2}{2}\right)ds
\]
with $\gamma(0)=0$. Note that $\gamma\left([0,\frac{1}{\epsilon}]\right) \subset \left[-\frac{1}{2}, \frac{1}{2}\right]$. 
Consider the following two cases.
\begin{enumerate}
    \item For any $s \in \left[0,\frac{1}{\epsilon}\right]$ such that $\gamma\left(s\right) \in \left[-\frac{1}{2}+\epsilon^\frac{1}{4}, \frac{1}{2}-\epsilon^\frac{1}{4}\right]$, there holds 
    \[f\left(\epsilon\gamma(s)\right) + W(\gamma(s))\geq \frac{1}{2} - \left(\frac{1}{2}-\epsilon^\frac{1}{4}\right)=\epsilon^\frac{1}{4}.
    \]
    \item For any $s \in \left[0,\frac{1}{\epsilon}\right]$ such that $\gamma\left(s\right) \in \left[-\frac{1}{2}, -\frac{1}{2}+\epsilon^\frac{1}{4}\right] \cup \left[\frac{1}{2}-\epsilon^\frac{1}{4}, \frac{1}{2}\right]$, there holds
    \[f\left(\epsilon\gamma(s)\right) + W(\gamma(s))\geq \left(\epsilon\left(\frac{1}{2} - \epsilon^\frac{1}{4} \right)\right)^\frac{1}{4}=\frac{1}{16}\epsilon^\frac{1}{4}-\epsilon^\frac{5}{16} \geq \frac{1}{32} \epsilon^\frac{1}{4},\]
    if $\epsilon\in(0,2^{-80})$.
\end{enumerate}

Therefore, 
\[
u^\epsilon(0,1)=\epsilon\int_0^\frac{1}{\epsilon} \left(f\left(\epsilon\gamma(s)\right) + W(\gamma(s))+\frac{\left|\dot{\gamma}(s)\right|^2}{2}\right)ds \geq \frac{1}{32} \epsilon^\frac{1}{4},
\]
that is,
\[u^\epsilon(0,1)-u(0,1) \geq \frac{1}{32} \epsilon^\frac{1}{4}.\]
\end{proof}

Finally, we illustrate the necessity of involving the time variable $t$ in \eqref{conclusion} and the discount coefficient $\lambda$ in \eqref{conclusion2} in the following example.

\begin{proposition}\label{exp:time}
Let $n=1$ and $H(x,y,p)=-f(x)-W(y)+\frac{1}{2}p^2$, where $f: \mathbb{R} \to \mathbb{R}$ is defined by
\[
f(x):= \left\{
\begin{aligned}
&|x|, \ \ \quad \text{if } |x| \leq 1,\\
& 1, \hspace{0.5mm}\qquad \, \text{if } |x| > 1,
\end{aligned}
\right.\] and
$W : \mathbb{T} \to \mathbb{R}$ is defined by
\[
W(y)=\frac{1}{2}-|y|, \quad \text{ if } y \in \left[-\frac{1}{2}, \frac{1}{2}\right].
\]

(i) Let $u^{\epsilon}$ be the unique solution to \eqref{eqn:ms}, and let $u$ be the unique solution to \eqref{eqn:eh} with $g\equiv0$, respectively. Then, for $\epsilon\in(0,1)$,
$$
u^{\epsilon}(0,t)-u(0,t)\geq\frac{1}{2}\epsilon t
$$
for all $t>0$.

(ii) Let $u^{\epsilon}$ be the unique solution to \eqref{discountepsilon}, and let $u$ be the unique solution to \eqref{discountbar}. Then, for $\lambda,\epsilon\in(0,1)$,
$$
u^{\epsilon}(0)-u(0)\geq \frac{\epsilon}{2\lambda}.
$$
\end{proposition}
\begin{proof}
We adopt the notations introduced in the proof of Proposition \ref{exp:nolip}. We give a proof of (i) and that of (ii) in order.

\medskip

(i) By the fact that $\overline{L}_1(v)\geq0$ for all $v\in\mathbb{R}$ and $\overline{L}_1(0)=0$, and by the optimal control formula of $u$,
\begin{equation*}
    \begin{aligned}
        u(x, t) = \inf \left\lbrace \int_0^t f\left(\eta(s)\right) +\overline{L}_1\left(\dot{\eta}(s)\right) ds : \eta\in \mathrm{AC}([0,t];\mathbb{R}),\ \eta(t) = x\right\rbrace,
    \end{aligned}
\end{equation*}
we conclude that $u(0,t)=0$ for all $t>0$.

On the other hand, note that the function $x\in\mathbb{R}\mapsto f(x)+W\left(\frac{x}{\varepsilon}\right)$ has its minimum value $\frac{1}{2}\epsilon$. Therefore, by the optimal control formula of $u^{\epsilon}$,
\begin{equation*}
    \begin{aligned}
        u^{\epsilon}(x, t) = \inf \left\lbrace \int_0^t f\left(\eta(s)\right)+ W\left(\frac{\eta(s)}{\epsilon}\right) +\frac{1}{2}|\dot{\eta}(s)|^2 ds : \eta\in \mathrm{AC}([0,t];\mathbb{R}),\ \eta(t) = x\right\rbrace,
    \end{aligned}
\end{equation*}
we see that $u^{\epsilon}(0,t)\geq\frac{1}{2}\epsilon t$.

\medskip

(ii) Due to the same reason, we see that $u(0)=0$. Also, by the optimal control formula of $u^{\epsilon}$,
\begin{align*}
u^\epsilon (x) &= \inf \left\lbrace \int_0^{\infty} e^{-\lambda s} \left(f\left(\eta(s)\right)+ W\left(\frac{\eta(s)}{\epsilon}\right) +\frac{1}{2}|\dot{\eta}(s)|^2\right)ds: \eta\in \mathrm{AC}([0,+\infty);\mathbb{R}^n),\ \eta(0)=x\right\rbrace\\
&\geq \frac{1}{2}\epsilon\int_0^{\infty} e^{-\lambda s} ds=\frac{\epsilon}{2\lambda}.
\end{align*}
This completes the proof.
\end{proof}

\section*{Appendix}

In this appendix, we prove Lemma \ref{lem:metric} with emphasis on the dependence on parameters. The goal is to show the constant $C>0$ appearing in the conclusion of the theorem is independent of the choice of $c\in\mathbb{R}^n$. Before we move into the proof, we set the following notation for convenience; for $c,x,y\in\mathbb{R}^n$ and $t>0$, we let
$$
m^c(t,x,y):=m^1_{c}(0,t,x,y)
$$
where the right-hand side $m^1_{c}(0,t,x,y)$ is as defined in Definition \ref{def:notation} with $\epsilon=1$.

First of all, from \cite[Lemma 3.2]{TranYu2022}, we see that for any $y\in\mathbb{R}^n,\ \epsilon,t>0$ with $|y|\leq M_0t$, there exists a constant $C=C(n,M_0,K_0)>0$ such that
$$
2m^c(t,0,y)\leq m^c(2t,0,2y)+C,
$$
which results in one direction of the conclusion, i.e.,
\[
m^{\epsilon}_c(0,t,a,b)\leq \overline{m}_c(0,t,a,b)+C \epsilon,
\]
for any $a, b \in \mathbb{R}^n,$ $\epsilon, t>0$ with $|b-a|\leq M_0t$. The independence of $C>0$ on $c\in\mathbb{R}^n$ is well shown by the argument of the proof of \cite[Lemma 3.2]{TranYu2022} together with \eqref{K_0H}, \eqref{K_0L}, which hold with a constant $K_0$ uniform in $c\in\mathbb{R}^n$ under the assumption (H1) in this paper. Thus, we skip the proof, and focus on the other direction instead.


Next, we show the other direction by verifying that for any $c\in\mathbb{R}^n$, and for any $y\in\mathbb{R}^n,\ \epsilon,t>0$ with $|y|\leq M_0t$, there exists a constant $C=C(n,M_0,K_0)>0$ such that
$$
m^c(2t,0,2y)\leq 2m^c(t,0,y)+C,
$$
which completes the proof of Lemma \ref{lem:metric} and also shows the independence of the constant on $c\in\mathbb{R}^n$.

\begin{lemma}
Assume {\rm(H1)-(H3)}. Let $M_0>0$, and let $K_0>0$ be a constant that satisfies in \eqref{K_0H}, \eqref{K_0L}. Then, there exists a constant $C=C(n,M_0,K_0)>0$ such that for any $c\in\mathbb{R}^n$, $t>0$ and any $y\in \mathbb{R}^n$ such that $|y|\leq M_0t$, we have
\[
m^c(2t,0,2y) \leq 2m^c(t,0,y)+C.
\]
\end{lemma}
\begin{proof}
Since $m^c(2t,0,2y) \leq m^c(t,0,y)+m^c(t,y,2y)$, it suffices to show that for some constant $C=C(n,M_0,K_0)>0$, it holds that
\[m^c(t,y,2y)\leq m^c(t,0,y)+C.\]

\begin{enumerate}
    \item If $t \leq 6$, then, by considering $\alpha:[0,t]\to\mathbb{R}^n$ defined by $\alpha(s)=y+\frac{s}{t}y$ for $s\in[0,t]$, we obtain, by \eqref{K_0L},
    \[
    m^c(t,y,2y)\leq \int_0^t \frac{1}{2}M_0^2+K_0 ds \leq 3M_0^2 + 6 K_0.
    \]
Also,
\[
m^c(t,0,y) \geq \int_0^t - K_0 ds\geq -6K_0.
\]
Hence, $m^c(t,y,2y) \leq m^c(t,0,y) +3M_0^2+12K_0$.
   \item If $t >6$, let $\zeta:[0,t]\to\mathbb{R}^n$ be an absolutely continuous curve with $\zeta(0)=0,\zeta(t)=y$ such that
   $$
   \int_0^tL(c,\zeta(s),\dot{\zeta}(s))ds\leq m^c(t,0,y)+1.
   $$
   By considering a straight line $\alpha:[0,t]\to\mathbb{R}^n$ defined by $\alpha(s)=\frac{s}{t}y$ for $s\in[0,t]$, we see that
   $$
   \int_0^tL(c,\zeta(s),\dot{\zeta}(s))ds\leq\int_0^t \frac{1}{2}M_0^2+K_0 ds+1 = \left(\frac{1}{2}M_0^2+K_0\right)t+1.
   $$

   We claim that there exists a number $d\in\left\{\frac{3}{2}k:0\leq k< \lfloor \frac{2}{3}t\rfloor\right\}$ such that
   $$
   \int_d^{d+\frac{3}{2}}L(c,\zeta(s),\dot{\zeta}(s))ds\leq M_0^2+3K_0+1.
   $$
   Otherwise, we would have
   $$
   \int_0^{\frac{3}{2}\lfloor\frac{2}{3}t\rfloor} L(c,\zeta(s),\dot{\zeta}(s))ds> \left\lfloor\frac{2}{3}t \right\rfloor\left(M_0^2+3K_0+1\right),
   $$
   which then lead to
   \begin{align*}
        \left(\frac{1}{2}M_0^2+K_0\right)t+1&\geq\int_0^{\frac{3}{2}\lfloor\frac{2}{3}t\rfloor}L(c,\zeta(s),\dot{\zeta}(s))ds+\int_{\frac{3}{2}\lfloor\frac{2}{3}t\rfloor}^{t}L(c,\zeta(s),\dot{\zeta}(s))ds\\
        &>\left(\frac{2}{3}t-1\right)\left(M_0^2+3K_0+1\right)-\left|t-\frac{3}{2}\left\lfloor\frac{2}{3}t\right\rfloor\right|K_0\\
        &>\left(\frac{2}{3}t-1\right)\left(M_0^2+3K_0+1\right)-3K_0.
   \end{align*}
   This is absurd for $t>6$.
   
   Let $w \in \mathbb{Z}^n$ such that $y-w \in [0,1]^n$. Define a new curve $\tilde{\zeta}:[0,t]\to \mathbb{R}^n$ by
   \begin{equation*}
   \tilde{\zeta}(s):=\left\{
       \begin{aligned}
        &w+\frac{\frac{1}{2}-s}{\frac{1}{2}}\left(y-w\right),\, \ \qquad \qquad \qquad \qquad 0\leq s\leq\frac{1}{2},\\
        &\zeta \left(s-\frac{1}{2}\right)+w, \qquad \qquad \qquad \qquad \qquad \frac{1}{2} \leq s \leq d+ \frac{1}{2},\\
        &\zeta \left(d+3\left(s-\left(d+\frac{1}{2}\right)\right)\right)+w,\, \ \ \qquad d+ \frac{1}{2}\leq s\leq d+1,\\
        &\zeta \left(s+\frac{1}{2}\right)+w, \qquad \qquad \qquad \qquad \qquad d+1 \leq s \leq t- \frac{1}{2},\\
        &y+w+\frac{s-t+\frac{1}{2}}{\frac{1}{2}} \left(y-w\right), \ \qquad \qquad \quad t-\frac{1}{2} \leq s \leq t.
       \end{aligned}
       \right.
   \end{equation*}
   Then, $m^c(t,y,2y)\leq\int_0^tL(c,\tilde{\zeta}(s),\dot{\tilde{\zeta}}(s))ds$ since $\tilde{\zeta}$ is an absolutely continuous curve from $y$ to $2y$. From the definition of $\tilde{\zeta}$, we see that
   $$
   \int_0^{\frac{1}{2}}L(c,\tilde{\zeta}(s),\dot{\tilde{\zeta}}(s))ds+\int_{t-\frac{1}{2}}^{t}L(c,\tilde{\zeta}(s),\dot{\tilde{\zeta}}(s))ds\leq 2n+K_0
   $$
   by \eqref{K_0L}, and that
   \begin{align*}
       \int_{\frac{1}{2}}^{d+\frac{1}{2}}L(c,\tilde{\zeta}(s),\dot{\tilde{\zeta}}(s))ds+\int_{d+1}^{t-\frac{1}{2}}L(c,\tilde{\zeta}(s),\dot{\tilde{\zeta}}(s))ds&=\int_0^dL(c,\zeta(s),\dot{\zeta}(s))ds+\int_{d+\frac{3}{2}}^{t}L(c,\zeta(s),\dot{\zeta}(s))ds\\
       &\leq m^c(t,0,y)+1-\int_d^{d+\frac{3}{2}}L(c,\zeta(s),\dot{\zeta}(s))ds\\
       &\leq m^c(t,0,y)+\frac{3}{2}K_0+1
   \end{align*}
    by the fact that $L(x,y,v)$ is periodic in $y$ and \eqref{K_0L} again. Finally, by the change of variables and by the choice of the number $d$, we get
    \begin{align*}
        \int_{d+\frac{1}{2}}^{d+1}L(c,\tilde{\zeta}(s),\dot{\tilde{\zeta}}(s))ds&=\frac{1}{3}\int_d^{d+\frac{3}{2}}L(c,\zeta(s),3\dot{\zeta}(s))ds\\
        &\leq\frac{1}{2}K_0+\frac{3}{2}\int_{d}^{d+\frac{3}{2}}\left|\dot{\zeta}(s)\right|^2ds\\
        &\leq\frac{7}{2}K_0+3\int_{d}^{d+\frac{3}{2}}L(c,\zeta(s),\dot{\zeta}(s))ds\leq 3M_0^2+\frac{25}{2}K_0+3.
    \end{align*}

    All in all, in the case when $t>6$, we see that there exists a constant $C=C(n,M_0,K_0)>0$ such that
    $$
    m^c(t,y,2y)\leq m^c(t,0,y)+C,
    $$
    and we complete the proof.
\end{enumerate}
\end{proof}

\section*{Acknowledgement}
The authors would like to thank Hung V. Tran for suggesting the problem, helpful conversations, and valuable advice.

\bibliographystyle{apalike}
\bibliography{ref}
\end{document}